\documentclass[review,12pt]{article}

\usepackage{hyperref,float,ragged2e,authblk}
\usepackage{amsmath,amssymb,amsthm,graphicx}

\newcommand{\noi}{\noindent}

\newtheorem{theorem}{Theorem}[section]

\newtheorem{observation}{Observation}[section]
\newtheorem{corollary}{Corollary}[section]

\begin{document}

%\begin{frontmatter}

\title{The Path Partition Conjecture is True and its Validity Yields
  Upper Bounds for Detour Chromatic Number and Star Chromatic Number} 


\author{G. Sethuraman\\
Department of Mathematics, Anna University\\
Chennai 600 025, INDIA\\
sethu@annauniv.edu}

\maketitle

\begin{abstract}
The detour order of a graph $G$, denoted $\tau(G)$, is the order of a
longest path in $G$. A partition $(A, B)$ of $V(G)$ such that
$\tau(\langle A \rangle) \leq a$ and $\tau(\langle B \rangle) \leq b$
is called an $(a, b)$-partition of $G$. A graph $G$ is called
$\tau$-partitionable if $G$ has an $(a, b)$-partition for every pair
$(a, b)$ of positive integers such that $a + b = \tau(G)$.  The
well-known Path Partition Conjecture states that every graph is
$\tau$-partitionable. In \cite{df07} Dunber and Frick have shown that
if every 2-connected graph is $\tau$-partitionable then every graph is
$\tau$-partitionable.  In this paper we show that every 2-connected
graph is $\tau$-partitionable. Thus, our result settles
the Path Partition Conjecture affirmatively. We prove the following two theorems as the implications of the validity of the Path Partition Conjecture.\\
{\bf Theorem 1:} For every graph $G$, $\chi_s(G) \leq \tau(G)$, where $\chi_s(G)$ is the star chromatic number of a graph $G$.
\newpage
The $n^{th}$ detour chromatic number of a graph $G$, denoted $\chi_n(G)$, is the minimum number of colours required for colouring the vertices of $G$ such that no path of order greater than $n$ is mono coloured. These chromatic numbers were introduced by Chartrand, Gellar and Hedetniemi\cite{cg68} as a generalization of vertex chromatic number $\chi(G)$.\\ 
{\bf Theorem 2:} For every graph $G$ and for every $n \geq 1$, $\chi_n(G) \leq \left\lceil \frac{\tau_n(G)}{n} \right\rceil$, where $\chi_n(G)$ denote the $n^{th}$ detour chromatic number.\\
Theorem 2 settles the conjecture of Frick and Bullock \cite{fb01} that $\chi_n(G) \leq \left\lceil \frac{\tau(G)}{n} \right\rceil$, for every graph $G$, for every $n \geq 1$, affirmatively.

\end{abstract}

%\begin{keyword}

%\MSC 05C15 \sep 05C70.
%\end{keyword}

%\end{frontmatter}

%\linenumbers

\section{Introduction}

All graphs considered here are simple, finite and undirected. 
Terms not defined here can be referred from the book
\cite{we02}. A longest path in a graph $G$ is called a detour of
$G$. The number of vertices in a detour of $G$ is called the detour
order of $G$ and is denoted by $\tau(G)$. A partition $(A, B)$ of
$V(G)$ such that $\tau(\langle A \rangle) \leq a$ and $\tau(\langle B
\rangle) \leq b$ is called an $(a, b)$-partition of $G$. If $G$ has an
$(a, b)$-partition for every pair $(a, b)$ of positive integers such
that $a + b = \tau(G)$, then we say that $G$ is $\tau$-partitionable.
The following conjecture is popularly known as the Path Partition
Conjecture.

\noi\textbf{Path Partition Conjecture:} {\it Every graph is
  $\tau$-partitionable}.

The Path Partition Conjecture was discussed by Lovasz and Mihok in
1981 in Szeged and treated in the theses \cite{ha84} and
\cite{vr86}. The Path Partition Conjecture first appeared in the
literature in 1983, in a paper by Laborde et al. \cite{lp82}. In 1995
Bondy \cite{bo95} posed the directed version of the Path Partition
Conjecture. In 2004, Aldred and Thomassen \cite{at04} disproved two stronger versions of the Path Partition Conjecture, known as the Path Kernel Conjecture \cite{bh97,mi85} and the Maximum
$P_n$-free Set Conjecture \cite{df04}. Similar partitions were studied for other
graph parameters too. Lovasz proved in \cite{lo66} that every graph is
$\Delta$-partitionable, where $\Delta$ denotes the maximum degree (A
graph $G$ is $\Delta$-partitionable if, for every pair $(a, b)$ of
positive integers satisfying $a + b = \Delta(G) - 1$, there exists a
partition $(A, B)$ of $V(G)$ such that $\Delta(\langle A \rangle) \leq
a$ and $\Delta(\langle B \rangle) \leq b$). For the results pertaining
to the Path Partition Conjecture and related conjectures refer
\cite{bd98,bh97,df99,df07,df04,fb01,fr13,ha84,lp82,mi85,se11,vr86,niel}. An
$n$-detour colouring of a graph $G$ is a colouring of the vertices of $G$ such
that no path of order greater than $n$ is monocoloured. The $n^{th}$
detour chromatic number of graph $G$, denoted by $\chi_n$, is the minimum number of
colours required for an $n$-detour colouring of a graph $G$. It is interesting to note that for a graph $G$, when $n=1$, $\chi_1(G)=\chi(G)$. These chromatic
numbers were introduced by Chartrand, Gellor and Hedetnimi \cite{cg68}
in 1968 as a generalization of vertex chromatic number.

If the Path Partition Conjecture is true, then the following
conjecture of Frick and Bullock \cite{fb01} is also true.

\noi\textbf{Frick-Bullock Conjecture:} $\chi_n(G) \leq \left\lceil
\frac{\tau(G)} {n} \right\rceil$ {\it for every graph $G$ and for
  every $n \geq 1$.}\\
Recently, Dunbar and Frick \cite{df07} proved the following theorem.

\begin{theorem}[Dunber and Frick \cite{df07}] \label{thm1.1}
If every 2-connected graph is \break $\tau$-partitionable then every graph
is $\tau$-partitionable.
\end{theorem}

In this paper we show that the Path Partition Conjecture is true for
every 2-connected graph. Thus, Theorem \ref{thm1.1} and our result
imply that the Path Partition Conjecture is true.  The validity of the
Path Partition Conjecture would imply the following Path Partition
Theorem.

\noi{\bf Path Partition Theorem.} {\it For every graph $G$ and for
  every $t$-tuple \break $(a_1, a_2, \dots, a_t)$ of positive
  integers with $a_1 + a_2 + \cdots + a_t = \tau(G)$ and $t \geq 1$,
  there exists a partition $(V_1, V_2, \dots, V_t)$ of $V(G)$
  such that $\tau(G(\langle V_i \rangle) \leq a_i$, for every $i$, $1 \leq i
  \leq t$.}\\
The Path Partition Theorem immediately implies that the Conjecture of
Frick and Bullock is true. The validity of Frick and Bullock Conjecture naturally implies the classical upper bound for the chromatic number of a graph $G$ that $\chi(G)=\chi_1(G) \leq \tau(G)$ proved by Gallai\cite{gall}.

A star colouring of a graph $G$ is a proper vertex colouring in which
every path on four vertices uses at least three distinct colours. The
star chromatic number of $G$ denoted by $\chi_s(G)$ is the least
number of colours needed to star color $G$. As a consequence of the
Path Partition Theorem, we have obtained an upper bound for the star
chromatic number.  More precisely, we show that $\chi_s(G) \leq
\tau(G)$ for every graph $G$.

\section{Main Result}

In this section we prove our main result that every 2-connected graph
is $\tau$-partitionable.

We use Whitney's Theorem on the characterization of 2-connected graph
in the proof of our main result given in Theorem \ref{thm2.2}.

An ear of a graph $G$ is a maximal path whose internal vertices have
degree 2 in $G$. An ear decomposition of $G$ is a decomposition $P_0,
P_1, \dots, P_k$ such that $P_0$ is a cycle and $P_i$ for $i \geq 1$
is an ear of $P_0 \cup P_1 \cup \dots \cup P_i$.

\begin{theorem}[Whitney \cite{wh}] A graph is 2-connected if and only
if it has an ear decomposition. Furthermore, every cycle in a
2-connected graph is the initial cycle in some ear decomposition.
\end{theorem}

\begin{theorem}\label{thm2.2}
Every 2-connected graph is $\tau$-partitionable.
\end{theorem}

\begin{proof}
Let $G$ be a 2-connected graph. By Whitney's Theorem there exists an
ear decomposition $S = \{P_0, P_1, \dots, P_n\}$, where $P_0$ is a
cycle and $P_i$ for $i \geq 1$ is an ear of $P_0 \cup P_1 \cup \dots
\cup P_i$.  We prove that $G$ is $\tau$-partitionable by induction on
$|S|$.

When $|S| = 1$, $S = \{P_0\}$. Then $G = P_0$. Thus, $G$ is a
cycle. As every cycle is $\tau$-partitionable, $G$ is
$\tau$-partitionable.  By induction, we assume that if $G$ is any
2-connected graph having an ear decomposition $S = \{P_0, P_1, \dots,
P_{k-1}\}$, that is, with $|S| = k$, then $G$ is $\tau$-partitionable.

Let $H$ be a 2-connected graph with an ear decomposition \break $S =
\{P_0, P_1, \dots, P_{k-1}, P_k\}$. That is, $|S| = k + 1$. We claim that $H$ is\\ $\tau$-partitionable. Let $(a,b)$ be a pair of positive integers with $a+b=\tau(H)$. Since $H$ is having the ear decomposition $S=\{P_0,P_1,\dots,P_{k-1},P_k\}$, $H$ can be considered as a 2-connected graph obtained from the 2-connected graph $G$ having the ear decomposition $S'=\{P_0,P_1,\dots,P_{k-1}\}$ by adding a new path (ear) $P_k:xv_1v_2\dots v_ry$ to $G$, where $x,y \in V(G)$ and $v_1,v_2,\dots,v_r$ are new vertices to $G$. As $G$ is a 2-connected graph having the ear decomposition $S'=\{P_0,P_1,\dots,P_{k-1}\}$ with $|S|=k$, by induction $G$ is $\tau$-partitionable. Let $(a_1,b_1)$ be a pair of positive integers such that $a_1 \leq a$, $b_1 \leq b$ with $\tau(G)=a_1+b_1$. Since $G$ is $\tau$-partitionable, there exists an $(a_1,b_1)$ partition $(A',B')$ of $V(G)$ such that $\tau(G(\langle A' \rangle)) \leq a_1$ and $\tau(G(\langle B' \rangle)) \leq b_1$. In order to prove our claim that $H$ is $\tau$-partitionable, we define an $(a,b)$-partition $(A,B)$ of $V(H)$ from the $(a_1,b_1)$ partition $(A',B')$ of $V(G)$ as well as using the path $P_k: xv_1v_2\dots v_ry$. The construction of an $(a,b)$-partition $(A,B)$ of $V(H)$ is given under three cases, depending on $r=0$, $r=1$ and $r \geq 2$, where $r$ is the number of new vertices in the path $P_k$.

\noi\textbf{Case 1.} $r = 0$

Then $P_k : xy$, where $x$ and $y$ are the vertices of $G$.\\
Thus, $H = G + xy$. This implies, $V(H) = V(G)$. 

\noi\textbf{Case 1.1.} Suppose $x$ and $y$ are in different parts of
the partition $(A^\prime, B^\prime)$ of 

\hspace{1.2cm} $V(G)$.

Then, as $x$ and $y$ are in different parts of the partition
$(A^\prime, B^\prime)$ of $V(G)$, the introduction of the new edge
$xy$ between the vertices $x$ and $y$ does not increase the length of
any path either in $G(\langle A^\prime \rangle)$ or in $G(\langle
B^\prime \rangle)$.  Further, as $V(H) = V(G)$, we have
$\tau(H(\langle A^\prime \rangle)) = \tau(G(\langle A^\prime \rangle))
\leq a_1 \leq a \ \text{ and } \ \tau(H(\langle B^\prime \rangle)) =
\tau(G(\langle B^\prime \rangle)) \leq b_1 \leq b.$
Thus, $(A^\prime, B^\prime)$ is a required $(a, b)$-partition of
$V(H)$.
\noi\textbf{Case 1.2.} Suppose $x$ and $y$ are in the same part of the
partition $(A^\prime, B^\prime)$ of 

\hspace{1.2cm} $V(G)$.

Without loss of generality, we assume that $x$ and $y$ are in
$A^\prime$. \\
Suppose $\tau(H(\langle A^\prime \rangle)) \leq a$.  Then,
as $\tau(H(\langle B^\prime \rangle)) \leq b_1 \leq b$, the
$(A^\prime, B^\prime)$ is a required $(a, b)$-partition of $V(H)$.\\
Suppose $\tau(H(\langle A^\prime \rangle)) > a$, then observe that the
addition of the edge $xy$ to $G$ has increased the order of some of
the longest paths (at least one longest path) in $H(\langle A^\prime
\rangle)$ from $a_1$ to $t = a_1 + k > a$, where $k \geq
1$. On the other hand, any path of order $t > a$ in $H(\langle A^\prime
\rangle)$ must contain the edge $xy$ also.

Let $P : u_1u_2u_3 \dots u_iu_{i+1} \dots u_{a_1} u_{a_{1+1}} \dots
u_t$ be any path of order $t > a$.  Then, note that the edge $xy = u_j
u_{j+1}$ for some $j$, $1 \leq j \leq t-1$ and $t \leq 2a_1$.

\begin{observation}
If we remove the vertex $u_{a+1}$ from the path $P$, then we obtain two
subpaths $u_1 u_2 \dots u_i u_{i+1} \dots u_{a-1} u_a$, say $P^\prime$
and $u_{a+2} u_{a+3} \dots u_{t-1} u_t$, say $P^{\prime\prime}$ of
$P$. The number of vertices in $P^\prime$ is exactly $a$ and the
number of vertices in $P^{\prime\prime}$ is $t-(a+1) \leq t-(a_1+1)
\leq 2a_1-a_1-1 = a_1-1 < a_1 \leq a$.
\end{observation}

\begin{observation}
Consider the subpath $Q : u_1 u_2 \dots u_{a-1} u_a u_{a+1}$ of
$P$. Then observe that the end vertex $u_{a+1}$ of $Q$ cannot be
adjacent to any of the end vertices of any path of order $b$ in the
induced subgraph $H(\langle B^\prime \rangle) = G(\langle B^\prime
\rangle)$ in $H$. 
\end{observation}

For, suppose $u_{a+1}$ is adjacent to an end vertex of a path, say $Z$
of order $b$ in $H(\langle B^\prime \rangle) = G(\langle B^\prime
\rangle)$. Let $Z = v_1 v_2 \dots v_b$.  Without loss of generality,
let $u_{a+1}$ be adjacent to $v_1$. Then, there exists a path $Q \cup
Z : u_1 u_2 \dots u_{a-1} u_a u_{a+1} v_1 v_2 \dots v_b$ of order
$a+b+1 > a+b = \tau(H)$, a contradiction (Similar contradiction hold
good if $u_{a+1}$ is adjacent $v_b$).

Let $\{R_0, R_1, \dots, R_t\}$ be the set of all paths in $H(\langle
A^\prime \rangle)$ of order at least $a+1$. For $1 \leq i \leq t$, let
$u_{a+1}^i$ denote the terminus vertex of the subpath of $R_i$ of
order $a+1$ and having its origin as the origin of $R_i$. Let
$\{u_{a+1}^{\alpha_1}, u_{a+1}^{\alpha_2}, \dots,
u_{a+1}^{\alpha_h}\}$ be the set of distinct vertices from the
vertices $u_{a+1}^1, u_{a+1}^2, \dots, u_{a+1}^t$, where $h \leq
t$. Suppose $\{u_{a+1}^{\alpha_1}, u_{a+1}^{\alpha_2}, \dots,
u_{a+1}^{\alpha_h}\}$ induces any path in $H(\langle A^\prime
\rangle)$.  Consider any such path $X : u_{a+1}^{\beta_1}
u_{a+1}^{\beta_2} \dots u_{a+1}^{\beta_c}$, where $\{\beta_1, \beta_2,
\dots, \beta_c\} \subseteq \{\alpha_1, \alpha_2, \dots,
\alpha_h\}$. Then, for $1 \leq i \leq c$, any vertex
$u_{a+1}^{\beta_i}$ divides the path $X$ into three subpaths
$u_{a+1}^{\beta_1} u_{a+1}^{\beta_2} \dots u_{a+1}^{\beta_{i-1}}$,
$u_{a+1}^{\beta_i}$, and $u_{a+1}^{\beta_{i+1}} u_{a+1}^{\beta_{i+2}}
\dots u_{a+1}^{\beta_c}$.

\begin{figure}[h]
\centering
  \includegraphics{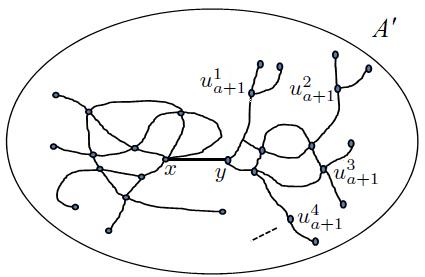}
  \caption{Structures of various paths of order $t \geq a+1$ in $A'$}
  \label{fig1}
\end{figure}

\noi\textbf{Claim 1.} For every $i$, $1 \leq i \leq h$, the vertex
$u_{a+1}^{\beta_i}$ cannot be adjacent to any of the end vertices of
any path of order greater than or equal to $b-q$ in $H(\langle
B^\prime \rangle)$, where $q = i-1$ or $c-i$.\\
First we ascertain  $b<q+1$ in Observation 2.3 then we prove the Claim 1.

\begin{observation}
$b \geq q+1$ \end{observation}
For, suppose $b<q+1$.
If $q = c-i$, then consider the path,
$$K = u_{a+1}^{\beta_i} u_{a+1}^{\beta_{i+1}} \dots u_{a+1}^{\beta_c} u_a^{\beta_c} u_{a-1}^{\beta_c} \dots u_2^{\beta_c} u_1^{\beta_c}$$
in $H(\langle A^\prime \rangle)$ having $1+c-i+a=1+q+a $ vertices. As $q+1 > b$, the path $K$ has at least $a+b+1$ vertices. This implies there exists a path of order at least $a+b+1$ in $H(\langle A^\prime \rangle)$. A contradiction to the fact that $\tau(H) = a+b$. Similarly, if        $q = i-1$, then consider the path, $$K^\prime = u_{a+1}^{\beta_i} u_{a+1}^{\beta_{i-1}} \dots              u_{a+1}^{\beta_1} u_a^{\beta_1} u_{a-1}^{\beta_1} \dots u_2^{\beta_1} u_1^{\beta_1}$$ 
in $H(\langle A^\prime \rangle)$ having $i+a = 1+q+a$ vertices. As $q+1 > b$, the path $K^\prime$ has at least $a+b+1$ vertices. This implies that there exists a path of order at least $a+b+1$ in in $H(\langle A^\prime \rangle)$. A contradiction to the fact that $\tau(H) = a+b$. Hence, $b \geq q+1$.\\

To prove Claim 1, we suppose $u_{a+1}^{\beta_i}$, for some $i$,
$1 \leq i \leq h$ is adjacent to an end vertex of a path of
order $l \geq b-q$ in $H(\langle B^\prime \rangle)$. Let $Y = w_1 w_2 w_3 \dots w_{l}$ be a path of
order $l \geq b-q$ in $H(\langle B^\prime \rangle)$ such that (without loss
of generality) $w_{l}$ is adjacent to the vertex $u_{a+1}^{\beta_i}$.

\noi\textbf{Case 1.2a.} $q = c-i$\\ Then consider the path $S = w_1
w_2 \dots w_{l} u_{a+1}^{\beta_i} u_{a+1}^{\beta_{i+1}} \dots
u_{a+1}^{\beta_c} u_a^{\beta_c} u_{a-1}^{\beta_c} \dots u_2^{\beta_c}
u_1^{\beta_c}$, where $u_1^{\beta_c} u_2^{\beta_c} \dots
u_{a}^{\beta_c} u_{a+1}^{\beta_c}$ is a subpath of $R_{\beta_c}$ of
order $a+1$ having the vertex $u_1^{\beta_c}$, the origin of
$R_{\beta_c}$ as its origin.  As $Y : w_1 w_2 w_3 \dots w_{l}$ is
the path in $H(\langle B^\prime \rangle)$ such that $w_{l}$ is
adjacent to $u_{a+1}^{\beta_i}$, it follows that $S$ is a path in $H$
having the order $l+1+c-i+a \geq b-q+1+q+a=b+a+1 > \tau(H)$, a contradiction.

\noi\textbf{Case 1.2b.} $q = i-1$

Then consider the path $S^\prime = w_1 w_2 \dots w_{l}
u_{a+1}^{\beta_i} u_{a+1}^{\beta_{i-1}} \dots u_{a+1}^{\beta_2}
u_{a+1}^{\beta_1} u_a^{\beta_1} u_{a-1}^{\beta_1}$ $\dots$
$u_2^{\beta_1} u_1^{\beta_1}$, where $u_1^{\beta_1} u_2^{\beta_1}
\dots u_a^{\beta_1} u_{a+1}^{\beta_1}$ is the subpath of $R_{\beta_1}$
of order $a+1$ having the vertex $u_1^{\beta_1}$, the origin of
$R_{\beta_1}$ as its origin. As $Y : w_1 w_2 w_3 \dots w_{l}$ is the
path in $H(\langle B^\prime \rangle)$ such that $w_{l}$ is adjacent
to $u_{a+1}^{\beta_i}$, it follows that $S^\prime$ is a path in $H$
having the order $l+1+i-1+a \geq b-q+1+q+a=b+a+1 > \tau(H)$, a
contradiction.\\ Hence the Claim 1.

Thus, it follows from the Claim 1 that 
\begin{align}\label{eq1}
\tau(H(\langle B^\prime \cup \{u_{a+1}^{\alpha_1}, u_{a+1}^{\alpha_2}, \dots,
u_{a+1}^{\alpha_h}\} \rangle)) \leq b
\end{align}
From Observation 1, it follows that
\begin{align}\label{eq2}
\tau(H(\langle A^\prime \backslash \{u_{a+1}^{\alpha_1}, u_{a+1}^{\alpha_2}, \dots,
u_{a+1}^{\alpha_h}\} \rangle)) \leq a
\end{align}
Let $A = A^\prime \backslash \{u_{a+1}^{\alpha_1}, u_{a+1}^{\alpha_2}, \dots,
u_{a+1}^{\alpha_h}\}$ and $B = B^\prime \cup \{u_{a+1}^{\alpha_1}, u_{a+1}^{\alpha_2}, \dots,
u_{a+1}^{\alpha_h}\}$. Then, from (\ref{eq1}) and (\ref{eq2}) it follows that
$\tau(H(\langle A \rangle) \leq a$ and $\tau(H(\langle B \rangle) \leq
b$.\\
Hence $(A, B)$ is a required $(a, b)$-partition of $H$.

\noi\textbf{Case 2.} $r = 1$\\
Then $P_k : xv_1y$.

\noi\textbf{Case 2.1.} Both $x$ and $y$ belong to the same partition
$A^\prime$ or $B^\prime$. \\ 
Without loss of generality, we assume that $x, y \in B^\prime$. That
is, $x, y \not\in A^\prime$. Then $(A^\prime \cup \{v_1\}, B^\prime)$
is a required $(a, b)$-partition of $V(H)$.\\
\noi\textbf{Case 2.2.} The vertices $x$ and $y$ belong to different
partitions $A^\prime$ and $B^\prime$.

Without loss of generality, we assume that $x \in A^\prime$ and $y \in
B^\prime$. If $x$ is not an end vertex of a path of order $a$ in
$H(\langle A^\prime \rangle)$, then $(A^\prime \cup \{v_1\},
B^\prime)$ is a required $(a, b)$-partition of $V(H)$. If $x$ is an end vertex
of a path of order $a$ in $H(\langle A^\prime \rangle)$, then $y$
cannot be an end-vertex of a path of order $b$ in $H(\langle B^\prime
\rangle)$ (otherwise, $H$ would have a path of order $a+b+1 >
\tau(H)$).  Therefore $(A^\prime, B^\prime \cup \{v_1\})$ is a required $(a,
b)$-partition of $V(H)$.

\noi\textbf{Case 3.} $r \geq 2$\\
Colour all vertices of $A^\prime$ with red colour and colour all the
vertices of $B^\prime$ with blue colour.  Since the vertices $x, y \in
V(G)$, they are coloured with either blue or red colour. Without loss
of generality, we assume that $x \in A^\prime$. Give $v_r$ the
alternate colour to that of the vertex $y$. As $x$ is coloured with
red colour, colour the vertex $v_1$ with blue colour. In general, for
$2 \leq i \leq r-1$, sequentially colour the vertex $v_i$ with the
alternate colour to the colour of the vertex $v_{i-1}$. Then observe
that $P_k$ contains no induced monochromatic subgraph of order greater
than 2 and no monochromatic path in $A^\prime$ or in $B^\prime$ can be
extended to include any of the vertices $v_1, v_2, \dots, v_r$ of
$P_k$.

Let $X_1$ be the set of all red coloured vertices of $P_k - \{x, y\}$
and let $X_2$ be the set of all blue coloured vertices of $P_k -\{x,
y\}$. Then $H(\langle A^\prime \cup X_1 \rangle) \leq a_1 \leq a$ and
$H(\langle B^\prime \cup X_2 \rangle) \leq b_1 \leq b$.  Hence
$(A^\prime \cup X_1, B^\prime \cup X_2)$ is a required $(a,
b)$-partition of $H$.\\
Thus, $H$ is $\tau$-partitionable. This completes the induction. Hence every 2-connected graph is $\tau$-partitionable.
\end{proof}

The following Corollary \ref{cor2.1} is an immediate consequence of
Theorem \ref{thm1.1} and Theorem \ref{thm2.2}.

\begin{corollary}\label{cor2.1}
Every graph is $\tau$-partitionable.
\end{corollary}

It is clear that Corollary \ref{cor2.1} settles the Path Partition
Conjecture affirmatively. Thus, ``{\bf the Path Partition Conjecture is true}''.

The following Theorem \ref{thm2.3} called ``Path Partition Theorem''
is a simple implication of Corollary \ref{cor2.1}.

\begin{theorem}[Path Partition Theorem] \label{thm2.3}
For every graph $G$ and for every $t$-tuple $(a_1, a_2, 
\dots, a_t)$ of positive integers with $a_1 + a_2 + \cdots + a_t =
\tau(G)$ and $t \geq 1$, there exists a partition $(V_1, V_2, 
\dots, V_t)$ of $V(G)$ such that $\tau(G(\langle V_i \rangle)) \leq
a_i$, for every $i$, $1 \leq i \leq t$.
\end{theorem}
\begin{proof}
Let $G$ be a graph. Consider any $t$-tuple $(a_1,a_2,\dots,a_t)$ of positive integers with $a_1+a_2+\dots+a_t=\tau(G)$, and $t \geq 1$. Then by Corollary 2.1, for the pair of positive integers $(a,b)$ with $a+b=\tau(G)$, where $a=a_1$ and $b=a_2+\cdots+a_t$, there exists a partition $(U_1,U_2)$ of $V(G)$ such that $\tau(G(\langle U_1\rangle)) \leq a = a_1$ and $\tau(G(\langle U_2 \rangle)) \leq b = a_2+a_3+\cdots+a_t$. Consider the graph $H=G(\langle U_2 \rangle)$. Then for the pair of positive integers $(c,d)$ with $c+d=\tau(H)=\tau(G(\langle U_2 \rangle))$, where $c=a_2$ and $d=a_3+a_4+\dots+a_t$, by Corollary 2.1, there exists a partition $(U_{21},U_{22})$ of $V(H)$ such that $\tau(H(\langle U_{21}\rangle )) \leq c =a_2$ and $\tau(H(\langle U_{22}\rangle )) \leq d=a_3+a_4+\dots+a_t$. As $H(\langle U_{21} \rangle)=G(\langle U_{21} \rangle)$ and $H(\langle U_{22} \rangle)=G(\langle U_{22} \rangle)$, we have $\tau(G(\langle U_{21} \rangle))\leq c = a_2$ and $\tau(G(\langle U_{22} \rangle))\leq d = a_3+a_4+\cdots+a_t$. Similarly, if we consider the pair of positive integers $(x,y)$ with $x+y=\tau(Q)$, where $Q=G(\langle U_{22} \rangle)$, $x=a_3$ and $y=a_4+a_5+\cdots+a_t$, by Corollary 2.1, we get a partition $(U_{31},U_{32})$ such that $\tau(G(\langle U_{31} \rangle)) \leq x =a_3$ and $\tau(G(\langle U_{32} \rangle)) \leq y =a_4+a_5+\cdots+a_t$. Continuing this process, finally we get a partition $(V_1,V_2,\cdots,V_t)$ of $V(G)$ such that $\tau(G(\langle V_i \rangle)) \leq a_i$, for every $i$, $1 \leq i \leq t$, where $V_1=U_1$, $V_2=U_{21}$, $V_3=U_{31}$ and so on. This completes the proof.     
\end{proof}
\begin{corollary}\label{cor2.2}
The $n^{th}$ detour chromatic number $\chi_n(G) \leq \left\lceil
\frac{\tau(G)} {n} \right\rceil$ for every graph $G$ and for every $n
\geq 1$.
\end{corollary}

\begin{proof}
Let $G$ be any graph. For every $n \geq 1$, consider the
$\frac{\tau(G)} {n}$-tuple $(n, n, \dots, n)$ if $\tau(G)$ is a
multiple of $n$, while if $\tau(G)$ is not a multiple of $n$, then
consider the $\left\lceil \frac{\tau(G)} {n} \right\rceil$-tuple $(n,
n, \dots, n, \nu)$, where $\nu = \tau(G)$ (mod $n$). Then, by Path
Partition Theorem, there exist a partition $(V_1, V_2, \dots,
V_t)$, where
\[t = \begin{cases} \frac{\tau(G)} {n} & \text{if } \tau(G) \text{ is
    a multiple of } n \\ \left\lceil \frac{\tau(G)} {n} \right\rceil &
  \text{if } \tau(G) \text{ is not a multiple of } n \end{cases}\]
such that $\tau(G(\langle V_i \rangle)) \leq n$, for every $i$, $1 \leq i \leq
t$. For each $i$, $1 \leq i \leq t$, assign the (distinct) colour $i$
to all the vertices in each $G(\langle V_i \rangle)$. Then every
monochromatic path in $G$ has the order at most $n$. Thus, $\chi_n(G)
\leq \left\lceil \frac{\tau(G)} {n} \right\rceil$.
\end{proof}

Corollary \ref{cor2.2} essentially ascertains that ``{\bf Frick-Bullock
Conjecture is true}''.\\
{\bf Remark 1:} It is clear from the definition of $\chi_n(G)$, when $n=1$, $\chi_1(G) = \chi(G)$. Thus, by Corollary \ref{cor2.2}, for a graph $G$, $\chi(G) = \chi_1(G) \leq \tau(G)$. This upper bound for the chromatic number of a graph $G$ that $\chi(G) \leq \tau(G)$ is the well known Gallai's Theorem \cite{gall}.
\section{An Upper Bound for Star Chromatic Number}

In this section we obtain an upper bound for star chromatic number as
a consequence of path partition theorem.

\begin{theorem}\label{star}
Let $G$ be a graph. Then the star chromatic number of $G$,
$\chi_s(G) \leq \tau(G)$.
\end{theorem}

\begin{proof}
First we prove the result for connected graphs, then the result follows naturally for the disconnected graphs. Let $G$ be a connected graph.

\noi\textbf{Claim 1:} There exists a proper $\tau(G)$-vertex colouring for $G$.

\noi Consider $\tau(G)$. If $\tau(G)$ is even, say $2k$, for some $k
\geq 1$, then consider the $k$-tuple $(2, 2, \dots, 2)$ with $2 + 2 +
2 + \cdots + 2 = 2k = \tau(G)$.  By Path Partition Theorem, there
exists a partition $(V_1, V_2, \dots, V_k)$ such that $\tau(G(\langle
V_i \rangle)) \leq 2$, for every $i$, $1 \leq i \leq k$.  Therefore,
every induced subgraph $G(\langle V_i \rangle)$, for $i$, $1 \leq i
\leq k$ is the union of a set of independent vertices and/or a set of
independent edges. Thus, it is clear that, for $i$, $1 \leq i \leq k$,
each $G(\langle V_i \rangle)$ is proper 2-vertex colourable. Properly
colour the vertices of each $G(\langle V_i \rangle)$ with a distinct
pair of colours $c_{i_1}$ and $c_{i_2}$, for $i$, $1 \leq i \leq
k$. Consequently, this proper 2-vertex colouring of $G(\langle V_i
\rangle)$, for all $i$, $1 \leq i \leq k$ induces a proper
$\tau(G)$-vertex colouring for the graph $G$. If $\tau(G)$ is odd, say $2k+1$, for some $k \geq 1$, then
consider the $k+1$-tuple $(2, 2, \dots, 2, 1)$ with $2 + 2 + 2 +
\cdots + 1 = 2k+1 = \tau(G)$.  Then by Path Partition Theorem there
exists a partition $(V_1, V_2, \dots, V_k, V_{k+1})$ such that
$\tau(G(\langle V_i \rangle)) \leq 2$, for every $i$, $1 \leq i \leq
k$ and $\tau(G(\langle V_{k+1} \rangle)) \leq 1$.  Consequently, the
vertices of each $G(\langle V_i \rangle)$ can be properly coloured
with a distinct pair of colours $c_{i_1}$ and $c_{i_2}$, for $i$, $1
\leq i \leq k$ and the vertices of $G(\langle V_{k+1} \rangle)$ are
colored properly with a distinct color $c_{(k+1)_1}$. Thus, this proper
2-vertex colouring of $G(\langle V_i \rangle)$, for all $i$, $1 \leq i \leq k$
and the proper 1 colouring of $G(\langle V_{k+1} \rangle)$ induce a
proper $\tau(G)$-vertex colouring for the graph $G$. Hence the Claim 1.

\noi\textbf{Claim 2:} $\chi_s(G) \leq \tau(G)$

\noi To prove Claim 2, we show that the vertices of every path of
order four is either coloured with 3 or 4 different colours by the above proper
$\tau(G)$-vertex colouring of $G$ or  if there exists a bicoloured
path of order four in $G$ by the above proper $\tau(G)$-vertex colouring of $G$,
then those vertices of such a bicoloured path of order four are
properly recoloured so that those vertices are coloured with at least
three different colours after the recolouring.

\begin{observation}
As 
\[\tau(G(\langle V_i \rangle)) \leq \begin{cases} 2, & \text{for } 1
  \leq i \leq k \\ 1, & \text{for } i = k+1 \text{ and } \tau(G)
  \text{ is odd} \end{cases}\] 
any path of order four in $G$ must contain vertices from at least two
of induced subgraphs $G(\langle V_i \rangle)$'s, where $1 \leq i \leq
\alpha$, and $\alpha = k$ when $\tau(G)$ is even, while when
$\tau(G)$ is odd, $\alpha = k+1$ \mbox{[}Hereafter $\alpha$ is either $k$ or $k+1$ depending on $\tau(G)$ is even or odd respectively\mbox{]}. If any path of order four of $G$
contains vertices from three or four of the induced subgraphs
$G(\langle V_i \rangle)$'s then such a path has vertices coloured with
three or four colours by the proper $\tau(G)$-vertex colouring of $G$.  Thus,
we consider only those paths of order four in $G$ having vertices from
exactly two of the induced subgraphs $G(\langle V_i \rangle)$'s, where
$1 \leq i \leq \alpha$, for recolouring if it is bicoloured.
\end{observation}

Consider any path $P$ of order
four in $G$ having at least one vertex (at most three vertices) in
$G(\langle V_i \rangle)$ for each $i$, $1 \leq i \leq \alpha$ and at least one vertex (at most three
vertices) in $G(\langle V_j \rangle)$, for every $j$, $1 \leq i < j \leq
\alpha$.

\noi\textbf{Case 1.} Suppose a path $P$ of order four in $G$ has one
vertex in $G(\langle V_i \rangle)$ 

\hspace{0.8cm} and three vertices in $G(\langle V_j \rangle)$, for $i,j$, $1 \leq i < j \leq \alpha$

\noi Then without loss of generality we assume that $u_{i_1}$ is one of
the vertices of $P$ which is in $G(\langle V_i \rangle)$ and we assume
$w_{j_1}, w_{j_2}$ and $w_{j_3}$ are the other three vertices of $P$
which are in $G(\langle V_j \rangle)$. Under this situation, in order
that the path $P$ is to be a path of order four with the vertices
$u_{i_1}, w_{j_1}, w_{j_2}, w_{j_3}$, two of the vertices from the three
vertices $w_{j_1}, w_{j_2}$ and $w_{j_3}$ in $G(\langle V_j \rangle)$
must be adjacent in $G(\langle V_j \rangle)$. Since $V_{k+1}$ is an independent set of vertices, $j \leq k$.  As the vertices of each induced subgraph $G(\langle
V_j \rangle)$ are properly coloured with 2 colours $c_{j_1}, c_{j_2}$,
for $j$, $1 \leq j \leq k$ by the proper $\tau(G)$-vertex colouring, those two adjacent vertices from the three vertices $w_{j_1}, w_{j_2}$ and $w_{j_3}$ in $G(\langle V_j \rangle)$
should have been coloured with two different colours $c_{j_1}, c_{j_2}$ by the $\tau(G)$-vertex colouring.
In $G(\langle V_i \rangle)$ each vertex is coloured with either
$c_{i_1}$ or $c_{i_2}$ by the proper $\tau(G)$ vertex colouring, the vertex $u_{i_1}$ is coloured with either $c_{i_1}$ or $c_{i_2}$ in $G(\langle V_i \rangle)$ by the proper $\tau(G)$-vertex colouring. This implies that the path $P$ of order four having the vertices $u_{i_1}, w_{j_1},
w_{j_2}$ and $w_{j_3}$ are coloured with at least three different
colours by the proper $\tau(G)$-vertex colouring of $G$.

\noi\textbf{Case 2} Suppose a path $P$ of order four in $G$ has
exactly two vertices in

\hspace{0.8cm} $G(\langle V_i \rangle)$ and has exactly two vertices
in $G(\langle V_j \rangle)$.

\noi Let $u_{i_1}$ and $u_{i_2}$ be the two vertices of $P$ in $G(\langle V_i
\rangle)$ and let $w_{j_1}$ and $w_{j_2}$ be the two vertices of $P$ in
$G(\langle V_j \rangle)$.

\noi\textbf{Case 2.1.} Suppose either $u_{i_1}, u_{i_2}$ are coloured
with two different colours 

\hspace{1.3cm} $c_{i_1}, c_{i_2}$ in $G(\langle V_i \rangle)$ or
$w_{j_1}, w_{j_2}$ are coloured with two different colours 

\hspace{1.3cm} $c_{j_1}, c_{j_2}$ in $G(\langle V_j \rangle)$ by the proper $\tau(G)$-vertex colouring.

\noi Then the vertices of the path $P$ of order four having the vertices
$u_{i_1}, u_{i_2}, w_{j_1}$ and $w_{j_2}$ are coloured with three or four
different colours by the proper $\tau(G)$-vertex colouring of $G$.

\noi\textbf{Case 2.2.} Suppose neither the vertices $u_{i_1}, u_{i_2}$
received different colours in 

\hspace{1.3cm} $G(\langle V_i \rangle)$ nor the vertices $w_{j_1},
w_{j_2}$ received different colours in 

\hspace{1.3cm} $G(\langle V_j \rangle)$ by the proper $\tau(G)$-vertex colouring.

\noi Then without loss of generality, we assume that $u_{i_1}, u_{i_2}$
received the same colour $c_{i_1}$ in $G(\langle V_i \rangle)$ and
without loss of generality, we assume that $w_{j_1}, w_{j_2}$ received
the same colour $c_{j_1}$ in $G(\langle V_i \rangle)$ by the $\tau(G)$-vertex colouring. As the vertices of $G$ are properly coloured, the vertices $u_{i_1}$ and $u_{i_2}$
should be non-adjacent in $G(\langle V_i \rangle)$ as well as the
vertices $w_{j_1}$ and $w_{j_2}$ should also be non-adjacent in
$G(\langle V_j \rangle)$. Since for every $h$, $1 \leq h \leq \alpha$, $\tau(G(\langle V_h \rangle)) \leq 2$, $G(\langle V_h \rangle)$ is the union of independent vertices and / or independent edges, every vertex in each $G(\langle V_h \rangle)$ is of degree either 0 or
1. Suppose either $u_{i_1}$ or
$u_{i_2}$ is of degree 0 in $G(\langle V_i \rangle)$. Then without loss
of generality, we assume that $u_{i_1}$ is of degree 0 in $G(\langle
V_i \rangle)$. Since $u_{i_1}$ is not adjacent to any vertex in
$G(\langle V_i \rangle)$, recolour the vertex $u_{i_1}$ with the colour
$c_{i_2}$ [Since vertices of $G(\langle V_i \rangle)$
  are properly coloured with either $c_{i_1}$ or $c_{i_2}$ colours, this recolouring is possible]. Thus,
after this recolouring, the vertices $u_{i_1}, u_{i_2}$, $w_{j_1}$ and
$w_{j_2}$ of the path $P$ have received three different colours. Hence,
we assume neither $u_{i_1}$ nor $u_{i_2}$ is of degree 0 in $G(\langle
V_i \rangle)$. Therefore, the degree of each of the vertices $u_{i_1}$ and
$u_{i_2}$ must be of degree 1 in $G(\langle V_i \rangle)$. As
vertices of each $G(\langle V_i \rangle)$ are properly coloured for
$i$, $1 \leq i \leq \alpha$ and as the vertices $u_{i_1}$ and $u_{i_2}$
are coloured with the same colour $c_{i_1}$ in $G(\langle V_i
\rangle)$, the vertices $u_{i_1}$ and $u_{i_2}$ must be non-adjacent in
$G(\langle V_i \rangle)$. Since the $deg(u_{i_1}) = 1$ in $G(\langle
V_i \rangle)$, the vertex $u_{i_1}$ should have an adjacent vertex
$u_{i_1}^\prime$ in $G(\langle V_i \rangle)$ and it should have been
coloured with the colour $c_{i_2}$ in $G(\langle V_i \rangle)$ by the proper $\tau(G)$-vertex colouring. For each $h$, $1 \leq h \leq k$, $\tau(G(\langle V_h \rangle) \leq 2$, the
edge $u_{i_1} u_{i_1}^\prime$ must be an independent edge in $G(\langle
V_i \rangle)$. Exchange the colours of $u_{i_1}$ and
$u_{i_1}^\prime$. Thus, after this recolouring (this exchange), the
vertex $u_{i_1}$ is coloured with $c_{i_2}$. Therefore, after the
recolouring the vertices $u_{i_1}$ and $u_{i_2}$ received two different
colours $c_{i_2}$ and $c_{i_1}$ respectively in $G(\langle V_i \rangle)$. As a result, the vertices
$u_{i_1}$, $u_{i_2}$, $w_{j_1}$ and $w_{j_2}$ have received three
different colours in $G(\langle V_i \cup V_j \rangle)$.  Hence the
path $P$ of order four having the four vertices $u_{i_1}$, $u_{i_2}$,
$w_{j_1}$, $w_{j_2}$ are coloured with three different colours in $G$
after the recolouring.

\noi\textbf{Case 3} Suppose the path $P$ of order four in $G$ has three
vertices in $G(\langle V_i \rangle)$

\hspace{0.8cm} and the remaining one vertex in $G(\langle V_j \rangle)$, for $i,j$, $1 \leq i < j \leq \alpha$.

\noi Without loss of generality, we assume that $w_{j_1}$ is one of the
vertices of $P$ which is in $G(\langle V_j \rangle)$ and we assume
$u_{i_1}$, $u_{i_2}$ and $u_{i_3}$ are the other three vertices of $P$
which are in $G(\langle V_i \rangle)$. Then as seen in Case 1, two of
the vertices from the three vertices $u_{i_1}$, $u_{i_2}$ and $u_{i_3}$
should have received two different colours $c_{i_1}$,
$c_{i_2}$ by the proper $\tau(G)$-vertex colouring. Consequently, the vertices of the Path $P$ should have received three or four different colours in $G$.

Thus, every path $P$ of order four in $G$ is either coloured with at
least three different colours by the proper $\tau(G)$-vertex colouring of $G$ or else if they are bicoloured by the proper $\tau(G)$-vertex colouring, then the
vertices of such a path $P$ can be recoloured as done in the above recolouring process
so that the vertices of $P$ are coloured with at least three different
colours.\\
Thus there exist a $\tau(G)$-star colouring for $G$. Hence, $\chi_s(G) \leq \tau(G)$. Hence Claim 2.

If $G$ is a disconnected graph with $t \geq 2$ components $G_1,G_2,\dots,G_t$. Then by Claim 2, $\chi_s(G_i) \leq \tau(G_i)$, for $i$, $1 \leq i \leq t$. Let $\smash{\displaystyle\max_{1 \leq i \leq t}}\,\, \chi_s(G_i) = \chi_s(G_k)$ for some $k$, $1 \leq k \leq t$. Since $\chi_s(G) = \smash{\displaystyle\max_{1 \leq i \leq t}}\,\,\chi_s(G_i)$, we have $\chi_s(G) = \chi_s(G_k) \leq \tau(G_k) \leq \smash{\displaystyle\max_{1 \leq i \leq t}}\,\, \tau(G_i) = \tau(G)$. Thus, $\chi_s(G) \leq \tau(G)$. 
This completes the proof.
\end{proof}

An acyclic colouring of $G$ is a proper vertex colouring of $G$ such
that no cycle of $G$ is bicoloured. Acyclic chromatic number of a
graph $G$, denoted $a(G)$ is the minimum of colours which are
necessary to acyclically colour $G$.

\begin{corollary}
Let $G$ be any graph. Then the acyclic chromatic number of $G$, $a(G)
\leq \tau(G)$.
\end{corollary}

\begin{proof}
For every graph $G$, $a(G) \leq \chi_s(G)$. By Theorem \ref{star}, we have
$\chi_s(G) \leq \tau(G)$ for any graph $G$. Thus, $a(G) \leq \tau(G)$,
for any graph $G$.
\end{proof}

\section{Discussion}
Path Partition Theorem is a beautiful and natural theorem and it significantly helped to get the upper bounds for chromatic number, star chromatic number and detour chromatic number. We believe that Path Partition Theorem can be significantly used for obtaining upper bounds of other different coloring related parameters too. In a general approach, understanding the following question will be interesting and significant too.
\begin{quotation}
What are the other graph parameters for which such partitions(like $\tau$-partition) can be obtained?
\end{quotation}  

\section*{Acknowledgement}

The author wishes to thank Professor Bill Jackson and Dr. Carol
Whitehead, University of London, for introducing the Path Partition
Conjecture to the author and for their continuous encouragement.

\end{document}